\documentclass[reqno,11pt]{amsart}

\usepackage[bookmarksnumbered]{hyperref}	
\usepackage[ngerman,USenglish]{babel}			
\usepackage{color}											
\usepackage[ascii]{inputenc}								
\usepackage{aliascnt}										
\usepackage{microtype}									
\usepackage[all]{hypcap}
\usepackage{tikz}											
\usepackage{verbatim}

\newcommand{\newjointcountertheorem}[3]{
	\newaliascnt{#1}{#2}
	\newtheorem{#1}[#1]{#3}
	\aliascntresetthe{#1}	
}

\newtheorem{thm}{Theorem}[section]
\newjointcountertheorem{satz}{thm}{Satz}
\newjointcountertheorem{lem}{thm}{Lemma}
\newjointcountertheorem{cor}{thm}{Corollary}
\newjointcountertheorem{prp}{thm}{Proposition}
\newjointcountertheorem{cnj}{thm}{Conjecture}
\newjointcountertheorem{que}{thm}{Question}
\newjointcountertheorem{fct}{thm}{Fact}
\theoremstyle{definition}
\newjointcountertheorem{dfn}{thm}{Definition}
\newjointcountertheorem{ntn}{thm}{Notation}
\newjointcountertheorem{rem}{thm}{Remark}
\newjointcountertheorem{nte}{thm}{Note}
\newjointcountertheorem{exl}{thm}{Example}
\newjointcountertheorem{rems}{thm}{Remarks}
\newjointcountertheorem{exls}{thm}{Examples}

\def\Snospace~{\S{}}

\renewcommand{\mathbb}{\mathbf}

\newcommand{\supp}{\mathrm{supp}}

\newcommand{\Q}{\mathbb{Q}}
\newcommand{\Z}{\mathbb{Z}}
\newcommand{\PP}{\mathbb{P}}
\newcommand{\N}{\mathbb{N}} 
\newcommand{\R}{\mathbb{R}}

\newcommand{\lra}{\longrightarrow}

\numberwithin{equation}{section}
\allowdisplaybreaks[4]							

\begin{document}
\title[{Maximal Percentages in P\'olya's Urn}]{Maximal Percentages in P\'olya's Urn}

	\author{Ernst Schulte-Geers and Wolfgang Stadje}
\address{Bundesamt f\"ur Sicherheit in der Informationstechnik (BSI), Godesberger Allee 185--189, 53175 Bonn, Germany} 
	\email{ernst.schulte-geers@bsi.bund.de}
	\address{Institute of Mathematics, University of Osnabr\"{u}ck, 49069 Osnabr\"{u}ck, Germany}
    \email{wstadje@uos.de}

\begin{abstract}
We show that the supremum of the successive percentages of red balls in P\'olya's urn model is 
almost surely rational, give the set of values that are taken with positive probability 
and derive several exact distributional
results for the all-time maximal percentage. 

\end{abstract}
	\maketitle


\section{Introduction}
The classical urn of Eggenberger-P\'olya \cite{pe} contains initially $r\geq 1$  red balls and $b\geq 1$ 
black balls. In the course of the drawing
process  in
each draw
one ball is taken from the urn (where each ball in the urn has the same chance of being drawn), and this ball 
and another $d\geq 1$ balls of the same color are put into the urn. The theory of P\'olya urn schemes 
is expounded in \cite{mahmoud}, where also an extensive bibliography can be found. 

Let $R_n$, $n\ge 1$, denote 
the number of red balls in the urn after the $n$th draw and set $R_0=r$.  
 The ratio $Z_n=R_n/(nd+r+b)$ gives the 
percentage of red balls among all balls in the urn 
at ``time'' $n$. Let $$S_{r,b}=\sup_{n \geq 0} Z_n$$ be the supremum of all successive percentages of red balls
during the entire drawing process.

In this paper we show that $S_{r,b}$ is attained almost surely 
and that $\PP(\hbox{$S_{r,b}$  is rational})=1$, thereby settling in the affirmative 
a conjecture of  
 Knuth (posed in the answer to problem 88 of  pre-fascicle 5a to volume $4B$ of 
{\em The Art of Computer Programming} \cite{knf5a}). The support of $S_{r,b}$ (i.e., 
the set of values that are taken 
 with positive 
probability) is also given.  
Moreover, we derive exact formulas for certain values of the 
distribution function of $S_{r,b}$, mainly in the case $d=1$. For general $d\ge 1 $ we show that  

\begin{align*}
\PP(S_{r,b}>(t-1)/t)&=\sum_{n=0}^\infty \dfrac{a+1}{n(t-1)+a+1}\\
& \hspace{0.3cm} {nt+a \choose n} \dfrac{B(n(t-1)+a+1+(r/d),n+(b/d))}{B(r/d,b/d)},  
\end{align*} 
where $(r+b)/r\leq t \in\mathbb{N}$, $a=\lfloor (b(t-1)-r)/d\rfloor$ 
and $B(\alpha, \beta)$ is the Beta function. For $d=1$ we have 
\begin{align*} 
\PP(S_{1,t-1}\leq 1/t)
=\left\{\begin{array}{ll} 1-\ln 2  &\mbox{ for } t=2\\
\Big[ (t-1)\left(1-\dfrac{1}{t}\right)^{t-2}-1\Big] \Big/(t-2) &\mbox { for } t>2 \end{array}\right. 
\end{align*} 
and  
	\begin{equation*}\\\begin{array}{ll}\PP(S_{1,1}&=(t-1)/t)=\dfrac{2t-3}{t}H(1-(1/t))-
\dfrac{t-2}{t}H(1-(2/t))\\
& \hspace{3.2cm} - \ \dfrac{t-2}{t-1}\\ \\ 
\PP(S_{1,1}&\le (t-1)/t)=(1-(1/t))H(1-(1/t)), 
\end{array}\end{equation*}
where $H(x)=\Psi(x+1)+\gamma$. Here 
$\Psi = \Gamma'/\Gamma $ denotes the Digamma function and $\gamma$ is Euler's constant. 
Another exact result in form of an infinite series is (again for $d=1$) 
$$\PP(S_{t-1,1} >(t-1)/t)  
=(t-1)\sum_{n=0}^\infty\frac{(nt)!}{(nt-n+1)!}\frac{((n+1)(t-1))!}{((n+1)t)!}. 
$$ 

In the course of our derivations we obtain a remarkable {\it equidistribution property} of the supremum 
$M(p) = \sup_{n\in \N} n^{-1}(B_1+\cdots + B_n)$ of the 
binomial random walk generated by 
i.i.d. Bernoulli variables $B_i$ (i.e., $\PP (B_i=1)=1-\PP(B_i=0)=p\in (0,1)$). 
For $t\in \N$ with $pt\le 1$ it is shown that 
\begin{align*} 
&\PP(M(p)\in (1/(k+1),1/k])= p/(1-p) \ \text{ for every } \ k \in \{1,\ldots, t-1 \}\\  
&\PP(M(p)\in (p, 1/t])=(1-tp)/(1-p).
\end{align*}  

\section{Used facts and related work}
In P\'olya's urn scheme let  $X_n=1$ if a red ball is drawn in the  $n$th 
draw and $X_n=0$ otherwise, and let $X=(X_1,X_2,\ldots)$ be the full 
sequence of these red ball indicators.

The following facts about P\'olya's urn are well-known \cite{black,freed}.  
\\

(1) $Z_n$ converges almost surely to a random variable $Z$, 
which has a Beta$(r/d,b/d)$ distribution on $(0,1)$. 
(Here and in the sequel Beta$(\alpha,\beta)$ denotes the Beta 
distribution  with parameters $\alpha, \beta > 0$). 

(2) Conditionally on $Z=z$ , $X_1, X_2, \ldots $ 
are independent, $\{0,1\}$-valued random variables with $\PP(X_i=1)=z$. 
\\

Thus the percentage $Z_n$ of red balls eventually tends to a random value $Z$, 
and if $Z_n$ ``exits'' to $Z=z$ the red ball indicators
behave like independent $0-1$ variables with ``success probability'' $z$. 
One can therefore try to average known results for the Bernoulli sequence to
obtain results for P\'olya's urn.

For a sequence of $\{-1,1\}$-valued Bernoulli variables $U_1,U_2,\ldots$ with $\PP(U_i=1)=1-\PP(U_i=-1)=p
\in (0,1)$ 
 the supremum  
$$M=\sup_{n\geq 1} n^{-1}(U_1+\ldots +U_n)$$ 
of the averages of their partial sums  
was considered in \cite{stad}. 
 It was shown that this supremum is attained and that   
\begin{enumerate}
\item $\PP(M\in (2p-1,1]\cap \mathbb{Q})=1$ 
\item  $\PP(M=x)>0$ for each $x\in (2p-1,1]\cap \mathbb{Q}$.
\end{enumerate}
Moreover, explicit formulas for the distribution of $M$ were derived.
Note that $M$ is one of the rare examples of a naturally occurring random variable 
taking every rational number in some interval with positive probability. 

 In the sequel we combine these results
to study the all-time maximal percentage of red balls in P\'olya's urn.

\section{Existence and Possible Values of Maxima}
Let $X$ be the P\'olya red ball indicator sequence.
We have already remarked that one may view this sequence as a randomized Bernoulli sequence.
To make this note self-contained 
 we include a short proof.  
\begin{thm}\label{thmKK}  
Let $Z$ and $Y=(Y_1,Y_2,\ldots ) $ be defined on some probability space such that: 
(a) $Z$ is $Beta(r/d,b/d)$-distributed. \\[0.1cm]
(b) Conditionally on $Z=z$ , $(Y_1,Y_2,\ldots )$ is an i.i.d. sequence of $\{0,1\}$-valued 
random variables 
with $\PP(Y_i=1)=1-\PP (Y_i=0)=z$.\\[0.1cm]
Then $Y\stackrel{d}= X$.
\end{thm}
\begin{proof} 
Let $(y_1,\ldots,y_n)\in\{0,1\}^n,\,s_n=y_1+\ldots +y_n$.
Then 
$$\PP(Y_1=y_1,\ldots,Y_n=y_n\mid Z=z)=z^{s_n}(1-z)^{n-s_n}.$$
Therefore, 
\begin{align*}
&\PP(Y_1=y_1,\ldots,Y_n=y_n)\\
&\hspace{0.9cm} = \dfrac{1}{B(r/d,\,b/d)} \int_0^1\,z^{(r/d)-1+s_n}(1-z)^{n+(b/d)-1-s_n}
\,dz\\& \hspace{0.9cm} =\dfrac{B((r/d)+s_n,\, n+(b/d)-s_n)}{B(r/d,\,b/d)}.
\end{align*} 
Hence a straightforward calculation (using the elementary properties of the Beta function) yields 
$\PP(Y_1=1)=r/(r+b)$ and  
$$\PP(Y_{n+1}=1 \mid Y_1=y_1,\ldots,Y_n=y_n)=\dfrac{r+s_nd}{nd+r+b},
$$   
so that the finite-dimensional distributions
of $Y$ coincide with those of $X$.
Thus, $Y\stackrel{d}= X$. 
\end{proof}

Since the red ball indicators in P\'olya's urn scheme are a randomized Bernoulli sequence
they have similar properties.

\begin{prp}  
In the situation of Theorem \ref{thmKK} let $S_n=\sum_{i=1}^n Y_i$ and $Z_n=(r+dS_n)/(nd+r+b)$.
Then $\lim_{n\to \infty}Z_n = Z$ a.s. and $\limsup_{n\to \infty} (S_n-nZ)=\infty$ a.s.
\end{prp}
\begin{proof}
Let $z\in (0,1)$. We only need to prove the assertion conditionally on $Z=z$.  
Since $Y$ is, conditionally on $Z=z$, a sequence of i.i.d. $0-1$ variables with probability $z$ for the value $1$, 
we have 
$$\PP(Z_n\lra z,\limsup_{n\to \infty}(S_n-nz)=\infty \mid Z=z)=1.$$
Note that $\PP(Z_n\lra z\,|\,Z=z)=1$ follows 
from the strong law of large numbers, 
and $\PP(\limsup_{n\to \infty}(S_n-nz)=\infty \mid Z=z)=1$ follows (e.g.) from the law of the iterated logarithm. 
\end{proof}

Let $M_{r,b}=\sup_{n\geq 1} Z_n$. We now use a variant of the proof of the corresponding property of $M$ 
in \cite{stad} 
to show that this supremum is almost surely attained.

\begin{prp} \label{prp0}
(a) The supremum in the definition of $M_{r,b}$ is almost surely attained. \\[0.1cm]
(b)$\;\PP( M_{r,b} \mbox { is rational})=1.$
\end{prp}
\begin{proof}
Again we condition on $Z=z$.
Since $\limsup_{n\to \infty} (S_n-nz)> b+1 $ a.s. we have that a.s. 
infinitely often  $(r+dS_n)-(nd+r+b)z>d(b+1)+r(1-z)-bz\ge 1+(r+b)(1-z)>0$, 
i.e., $Z_n>z$ for infinitely many $n$ a.s., and since $Z_n\longrightarrow z$ a.s.
the supremum is a.s. attained. It is then obviously rational. 
\end{proof}

Define 
\begin{align*} 
Q&=\Big\{ \max_{1\le j\le n} \dfrac{r+i_jd}{r+b+jd}\ \Big| \  n\in \N, i_1,\ldots,i_n\in \N \cup \{0\}, \\ 
&\hspace{1.5cm} i_j\le j \text{ for all } j\in \{1,\ldots,n\}, 0\le i_1\le i_2\le \ldots \le i_n\Big\}. 
\end{align*} 
Clearly,  $Q$ is the set of {\it all} possible maximal percentages for finite sequences of draws from the urn.  
It follows from Proposition \ref{prp0} that $\PP(M_{r,b}\in Q)=1$. Now we show 
\begin{prp}\label{supp} 
\ $\PP(M_{r,b} = q) >0 \ \text{ for all }\ q\in Q.$ 
\end{prp}
\begin{proof} 
Fix an element $q$ of $Q$ and corresponding $m \in \N$ and nonnegative integers $ i_1\le i_2\le \ldots \le i_m $ 
such that  $i_j \le j$ for all $j\in \{1,\ldots,m\}$ and 
$$q=\max_{1\le j\le m} \dfrac{r+i_jd}{r+b+jd}.$$
Let the maximum be attained at $n\in \{1,\ldots,m\}$. Then we have  
$$q = \dfrac{r+i_{n}d}{r+b+nd}= \max_{1\le j\le n} \dfrac{r+i_jd}{r+b+jd}.$$ 
Let $E$ be the event 
that $R_j=i_jd$ for $j=1,\ldots,n$ and $M_{r,b}=q$ 
(thus on this event the percentage after the $n$th draw is the largest among the first $n$ 
ones). It remains to show that $\PP(E)>0$.

Write $\PP(E\mid Z=z)$ in the form
\begin{align*} 
\PP(E \mid Z=z)&=\PP \Big(R_j=i_jd \text{ for } j=1,\ldots, n 
 \text{ and } \\
&  
\hspace{0.5cm}  \sup_{k\in \N} \dfrac{r+i_nd+(Y_1'+\ldots +Y_k')d}{r+b+nd+kd}\le 
\dfrac{r+i_{n}d}{r+b+nd} \ \Big| Z=z\Big)\\
&  
=\PP (R_j=i_jd \text{ for } j=1,\ldots, n\mid Z=z)
\\
&  
\hspace{0.5cm} \times \PP\Big( \sup_{k\in \N} \dfrac{r+i_nd+(Y_1'+\ldots +Y_k')d}{r+b+nd+kd}\le 
\dfrac{r+i_{n}d}{r+b+nd}\Big),   
\end{align*} 
where $Y_1',Y_2',\ldots$ is an i.i.d. sequence with $\PP(Y_i'=1)=1-\PP(Y_i'=0)=z$ which is independent 
of $Y$ and $Z$. The first factor on the righthand side is obviously positive. Next, the inequality 
$$\sup_{k\in \N} \dfrac{r+i_nd+(Y_1'+\ldots +Y_k')d}{r+b+nd+kd}\le \dfrac{r+i_{n}d}{r+b+nd}$$ 
holds if and only if
$$ \sup_{k\in \N} \dfrac{Y_1'+\ldots +Y_k'}{k} \le \dfrac{r+i_nd}{r+b+nd}.$$
By \cite{stad}, this latter inequality occurs with positive 
probability if $$\PP (Y_1'=1)< (r+i_nd)/(r+b+nd).$$ 
It follows that $\PP(E\mid Z=z)>0$ for $z\in (0,  (r+i_nd)/(r+b+nd))$.  Hence, 
\begin{align*} 
\PP(E)\ge \int_0^{\min [1,(r+i_nd)/(r+b+nd)]}\PP(E \mid Z=z) \beta_{r/d,b/d}(z)\ dz>0,   
\end{align*} 
where $\beta_{r/b,b/d}(z)$ is the density of Beta$(r/b,b/d)$. The proof is complete. 
\end{proof} 

\begin{rem}\label{support} 
Proposition \ref{supp} can also be formulated in the following form. Let supp$(U)$ denote 
the support of the random variable $U$. Let 
$q_{r,b}(i,n) = (r+id)/(r+b+nd)$ for $i,n \in \N \cup \{0\}$. Then 
\begin{align}\label{Srb} 
\text{supp}(S_{r,b}) = \{ q_{r,b}(i,n) \mid 0\le i \le n, q_{r,b}(i,n) \ge r/(r+b)\} 
\end{align}
and
\begin{equation}\label{Mrb}
\text{supp}(M_{r,b}) = \text{supp}(S_{r,b+d}) \cup \text{supp}(S_{r+d,b}).
\end{equation}   
(\ref{Mrb}) is obvious, and 
(\ref{Srb}) is proved in the same way as Proposition \ref{supp} (consider sequences of draws
beginning with $X_1=\ldots =X_{n-i}=0, X_{n-i+1}=\ldots, X_n=1$). 

In particular, for $d=1$
we obtain $\supp(S_{r,b})=[r/(r+b),1)\cap \Q$.  
\end{rem} 
\section{Some Closed-Form Results on Maximal Percentages for the Binomial Random Walk}

In this section we only consider the case $d=1$. 
We start with two exact results for the binomial random walk, 
which may be of independent interest. Although they are consequences of the t-ballot theorems
they apparently have not been formulated in this form before.  

In the sequel fix $0<p<1,\, q=1-p$ and let $Y_1,Y_2,...$ be i.i.d. $0-1$ random 
variables with $\PP(Y_i=1)=p$. Introduce their partial sums $S_n=Y_1+\cdots+Y_n$ and define, 
for $r,b\in \mathbb{N}\cup\{0\}$,   
$$M_{r,b}(p)=\sup_{n\geq 1} \frac{r+S_n}{r+b+n} \mbox{ and } M(p)=M_{0,0}(p)$$
$$L_{r,b}(p)=\inf_{n\geq 1} \frac{r+S_n}{r+b+n} \mbox{ and } L(p)=L_{0,0}(p).$$ 
Since 
$\PP(M_{r,b}(p)\leq x)=\PP(L_{b,r}(q)\geq 1-x)$ for $0\leq x \leq 1$ 
the distribution function of $L_{b,r}(q)$ can be obtained from that of $M_{r,b}(p)$.\\
 Clearly $M_{r,b}(p)\leq x$ if and only if $S_n-nx \leq bx-(1-x)r$ for all $n\in\mathbb{N}$.\\ 
Thus if $x\in \mathbb{Q}\cap(0,1)$,  say $x=s/t$ for positive integers $s$ and $t$ with $\gcd(s,t)=1$, 
the value of $\PP(M_{r,b}(p)\leq s/t)$ depends on $s,t$ only through the integer value $m=sb-(t-s)r$.\\
For $m\in \Z_+ $ we define  $a_m=\PP(tS_n-ns\leq m \mbox{ for all } n\in\mathbb{N})$.;

\begin{rem}  
For two integers $s,t$ satisfying $(s,t)=1$, $p<s/t<1$ the following can be shown 
(along the lines of \cite{stad}):
	\begin{enumerate}
		\item[(a)] The sequence $a_0,a_1,a_2,\ldots$ has a rational generating function; it is given by 
$q(\sum_{j=0}^{s-1} a_jz^j)/(pz^t-z^s+q)$. This function is regular for $|z|<1$. 
\item[(b)] The denominator
	$f(z)=pz^t-z^s+q$  has only simple roots: exactly $s-1$ roots $z_1,\ldots,z_{s-1}$ 
inside the unit disk, $z_{s}=1$, and exactly $t-s$ roots $z_{s+1},\ldots,z_{t}$ outside the unit disk.
\end{enumerate}
Thus $a_m$ can be written as a linear combination of $1$ and the $(m+1)$th 
negative powers of the roots of $f(z)$ outside the unit disk. However, explicit expressions are
hard to come by. 
\end{rem}

Let $t\geq 2$ be an integer. In the sequel we give expressions for the probabilities 
$\PP(M(p)\leq (t-1)/t)$  and $\PP(M(p)\leq 1/t)$. 

For the first case we need the $t-$ary tree function $T_t(z)$.
 Its power series is given by 
$$T_t(z)=\sum_{n=0}^\infty {nt \choose n} \frac{z^n}{n(t-1)+1}.$$ 
This power series converges for $|z|<\frac{1}{t}(1-\frac{1}{t})^{t-1}$
and for $z=\frac{1}{t}(1-\frac{1}{t})^{t-1}$
and represents there the unique solution of  the implicit equation $T_t(z)=1+z\,T_t(z)^t$.  
(Thus $1/T_t$ is the inverse function of $y\mapsto y^{t-1}(1-y)$ for $|1-y|<1/t$.) We define 
$$R_t(p)=pT_t(qp ^{t-1}).$$

We recall the following fact from path counting combinatorics (``$t$-ballot numbers").
(See e.g. \cite{kn4a}, Problem 26 of Section  7.2.1.6., for an equivalent statement.)
\begin{lem}\label{lemtbn}   
Let $a\in \mathbb{N}\cup\{0\}$. The number of paths, with steps  in 
$\{(0,1), (1,0)\}$,  from $(0,0)$ to $(n,n(t-1)+a)$ whose points lie
on or below the line $y=a+x(t-1)$ is the coefficient $[z^n] T_t(z)^{a+1}$. 
We have 
$$[z^n] T_t(z)^{a+1}=
\dfrac{a+1}{n(t-1)+a+1}{\displaystyle {nt+a \choose n}}.
$$ 
\end{lem}
 
\begin{prp}\label{prp1}  
If  $m=b(t-1)-r\geq 0$, we have  
\begin{eqnarray}{\mbox{  }}
 \PP(M_{r,b}(p)\leq (t-1)/t)&=&1-R_t(p)^{m+1}\\ \\ 
 \PP(M_{r,b}(p)=(t-1)/t)&=&\left\{\begin{array}{ll}  p-R_t(p)+qR_t(p)^{t-1}\  &\mbox{for } m=0\\ 
 \\ 
\big(1-R_t(p)\big)R_t(p)^m \ &\mbox {for } m\geq 1 \end{array}\right.
\end{eqnarray}
\end{prp}
\begin{proof} (1) Consider $\PP(M_{r,b}(p)>(t-1)/t)$. Start a lattice path $R_n=r+S_n, B_n=b+n-S_n$ in $(r,b)$ 
by stepping $(1,0)$ if $Y_i=1$ , $(0,1)$ if $Y_i=0$.
Since $M_{r,b}(p)>(t-1)/t$, there must be a smallest 
$j\geq b$ such that $R_\ell=j(t-1)+1, B_\ell=j$ and $(t-1)B_i\geq R_i$ for $i\leq \ell=jt-r-b+1$.\\
Equivalently, $j-b$ is the first position where 
the lattice path $(n-S_n,S_n)$ starting at $(0,0)$ steps at time $\ell$ for the first time 
above the line $y=(t-1)x+m$. Call this event $A_j$. Clearly the last step is $Y_\ell=1$, appended
to a path from $(0,0)$ to $(j-b,j(t-1)-r)=(j-b,(j-b)(t-1)+(t-1)b-r)$ 
of the type considered in the lemma above. 
Since each path from $(0,0)$ to $(j-b, j(t-1)-r+1)$ has the same 
probability $p^{j(t-1)-r+1}q^{j-b}$, we get 
\begin{align*} 
\PP(A_j)&=[z^{j-b}] T_t(z)^{m+1} 
p^{j(t-1)-r+1}q^{j-b}\\&=p^{m+1}[z^{j-b}]T_t(z)^{m+1} p^{(j-b)(t-1)}q^{j-b}. 
\end{align*}  
Since $\{M_{r,b}(p)>(t-1)/t\}$ is the disjoint union of the $A_j$, the
first assertion follows.\\
(2) For $k\geq 0$ let $b_k=\PP(tS_n<n(t-1) +k)  \mbox{ for all } n\in\mathbb{N})$. 
Conditioning with respect to $Y_1$ then shows that
$b_0=qa_{t-2}$ and $b_j=a_{j-1}$ for $j\geq 1$. Since $\PP(M_{r,b}(p)=(t-1)/t)=a_m-b_m$ the second assertion 
follows from the result in (1). 
\end{proof}

\begin{rems}\label{rem1}   
(a) The path counting argument above is due to Knuth, 
who used it to determine $\PP(S_{1,1}>(t-1)/t)$ in P\'olya's urn (see next section).
Alternatively the results could be obtained using the theorem from 
\cite{stad} given below, but in a more laborious way.\\
(b) In the derivation no restrictions on $p$ (other than $0<p<1$) were imposed, the formulas 
above are thus also valid
for $p\geq (t-1)/t$. In fact, by the remark before Lemma \ref{lemtbn}  we have 
$R_t(p)=1$ for $p\geq (t-1)/t$ and the corresponding probabilities
in Proposition \ref{prp1} evaluate to 0 (as they must do in view of the strong law of large numbers).\\
(c) $R_t(p)\lra p$ for $t\lra \infty$. More precisely, $\lim_{t\to \infty}q^{-1}p^{-t}[R_t(p)-p]=1$.
\end{rems} 

\begin{exls}\label{ex1}   
(a) Let $t=2$, $p<1/2$. Then $R_2(p)=\min(1,p/q)$ and  we recover the well-known facts that for $m=0$
$$\PP(M(p)> 1/2)=p/q \text{ and } \PP(M(p)=1/2)=(q-p)p/q.$$ 
and that for $m\geq 1$ 
$$\PP(M_{r,b}(p)\geq 1/2)=(p/q)^m \text{ and } \PP(M_{r,b}(p)=1/2)= (p/q)^m(q-p)/q.$$ 
(b) 
Let $p=1/2$. Then $R_2(p)=1$ and  
\begin{align*} 
&R_3(p)\approx 0.618034, R_4(p)\approx 0.543689, R_5(p)\approx 0.518790\\  & 
R_6(p)\approx 0.50866,
R_7(p)\approx 0.504138.
\end{align*} 
For example we have 
$$\PP \Big( \sup_{n\geq 1} \, \dfrac{S_n}{n}>\dfrac{6}{7}\Big)\approx  0.504138 \mbox{ and }
\PP \Big( \sup_{n\geq 1} \, \dfrac{S_n}{n+1}>\dfrac{2}{3}\Big)\approx 0.381937$$ 
\end{exls} 

Let us now turn to the second probability 
$\PP(M(p)\leq 1/t)$. We use the following result of \cite{stad}.

\begin{thm}\label{thmSta}  
Let $s,r \in\mathbb{Z}\setminus\{0\}$, $s>0$, $|r|<s$. The polynomial 
\begin{equation}\label{eqpol}p z^{2s}- z^{s+r} +q\end{equation} 
has exactly $s-r$ simple roots $z_{s+r},\ldots,z_{2s-1}$ outside the unit
disk. If $p<\tfrac{r+s}{2s}$, we have 
\begin{eqnarray} \PP(M(p)\leq \dfrac{r+s}{2s-1})&=&\prod_{i=r+s}^{2s-1}(1-z_i^{-1})\\
\PP(M(p)=\dfrac{r+s}{2s-1})&=&\prod_{i=r+s}^{2s-1}(1-z_i^{-1})+p\,\prod_{i=r+s}^{2s-1}(1-z_i). 
                  \end{eqnarray}
\end{thm}

\begin{rems}   
(a) In \cite{stad} this theorem is proved 
under the additional assumption $(r,s)=1$, but it remains valid if $(r,s)=k>1$. 
To see this let $u=z^k$. As a function of $u$ 
the polynomial \eqref{eqpol} has $(s-r)/k$ simple roots $u_i$ 
outside the unit disk, say $u_1,\ldots,u_{(s-r)/k}$, 
by the theorem above. 
If $\eta$ is a primitive $k$th root of unity, then as a function of $z$ 
the polynomial \eqref{eqpol}
has exactly the $s-r$ simple roots $z_{i,j}=\eta^j |u_i|^{1/k}, j=0,\ldots k-1$ and $i=1,\ldots,(s-r)/k$, 
outside the unit disk, and since  $\prod_{j=0}^{k-1} (1-z_{i,j}^{-1})=(1-u_i^{-1})$ and
$\prod_{j=0}^{k-1} (1-z_{i,j})=(1-u_i)$ the formulas above remain valid.\\
(b) In \cite{stad} also formulas for the corresponding $a_m$ in terms of the roots of \eqref{eqpol} 
were provided, but we do not use them here. 
\end{rems}

\begin{prp}\label{prp2}  
Let $p<1/t$. If  $m=b-(t-1)r\in \{0,\ldots,t-1\}$, we have   
\begin{eqnarray}{\mbox{  }}
	\label{ed1}\; \PP(M_{r,b}(p)\leq 1/t)&=&(1-tp)q^{-(m+1)}\\
	\label{ed2}\; \PP(M_{r,b}(p)=1/t)&=&(1-tp)q^{-(m+1)}p. 
\end{eqnarray}
For $m\geq t$ the $a_m$ can be computed by the recursion
\begin{equation} 
qa_k=a_{k-1}-pa_{k-t}. 
\end{equation}
For $p\geq 1/t$ the probabilities in \eqref{ed1} and \eqref{ed2} are $0$. 
\end{prp}
\begin{proof} (1) Let $m=0$ and set $s=t,\,r= -(t-2)$ in Theorem \ref{thmSta}. The polynomial
$pz^{2t}-z^2+q$ has $2t-2$ roots $z_2,\ldots,z_{2t-1}$ outside the unit disk, and
$a_0=\prod_{i=2}^{2t-1}(1-\tfrac{1}{z_i})$. Equivalently, since $+z_i,-z_i$ are roots, we have 
$a_0=\prod_{i=1}^{t-1} (1-\tfrac{1}{y_i})$ where $y_1,\ldots,y_{t-1}$ are the roots of $py^{t-1}-y+q$ 
outside the unit disk. Let $y=1/(1-u)$. Then we can write $a_0=\prod_{i=1}^{t-1}u_i$ where $u_1,\ldots,u_{t-1}$
are the non-zero roots of the polynomial $g(u)=p-(1-u)^{t-1}+q(1-u)^t$. Thus $qa_0=-g^\prime(0)=1-pt$.
Furthermore, since $\prod_{i=2}^{2t-2}z_i=-q/p$ we find from the second formula
$\PP(M(p)=1/t)=a_0-qa_0=pa_0$.\\
(2) By conditioning on $Y_1$ we find that $a_k=qa_{k+1}$ for $k=0,\ldots,t-2$, and 
$a_k=pa_{k-t+1}+qa_{k+1}$. The rest is straightforward.
\end{proof}

\begin{rems}   
(a) In the case $m=0$ and  $p\leq 1/t$, we find a curious equidistribution property: 
the maximum $M(p)$ lies in any of the intervals
$(1/(k+1),1/k]$ (for $1\leq k\leq t-1$) with the same probability $p/q$ (and it lies 
in the interval $(p,1/t]$ with probability $1-(t-1)p/q$).\\
(b) The result for $a_0$ can also be shown to follow from a path counting result 
(Barbier's theorem): the number of paths with step set $\{(0,1),(1,0)\}$ from $(0,0)$ to $(k,n)$
with $n>tk$ that never touch the line $y=tx$ except at $(0,0)$ is equal to 
$(n-tk) {n+k\choose n}/(n+k)$.
\end{rems}

\section{Distributional results for P\'olya's urn}

Returning to P\'olya's urn, recall that $S_{r,b}=\sup_{n\geq 0} Z_n=\max\{\frac{r}{r+b},\,M_{r,b}\}$ 
and let $I_{r,b}=\inf_{n\geq 0} Z_n=\min\{\frac{r}{r+b},\,L_{r,b}\}$.
We now give some formulas for the distribution function of $S_{r,b}$ and $I_{r,b}$ for special values. 

We introduce  
the generalized harmonic number function 
$$H(x)=\sum_{n\geq 1}\Big(\frac{1}{n}-\frac{1}{n+x}\Big)
=\Psi(x+1)+\gamma, \ \ x\in \R \backslash \{-1,-2,-3,\ldots\}$$ where 
$\Psi$ is the Digamma function 
(the logarithmic derivative of the $\Gamma$-function) and $\gamma$ is Euler's constant. 
For positive integers $p,q$ with $p<q$ Gauss has shown (see e.g. Problem 19 of Section 1.2.9 in \cite{kn1}) that
$$H(p/q)=\dfrac{q}{p}-\dfrac{\pi}{2}\cot \left(\frac{p}{q}\pi\right)
-\ln(2q)+2\sum_{1\leq n <q/2}\cos \left(\dfrac{2pn}{q}\pi\right)\,\ln\,\sin \left(\dfrac{n}{q}\pi\right).$$ 
Thus $H(p/q)$ can be expressed in terms of finitely many elementary functions. 

Let us first consider
the distribution function of $S_{r,b}$ at $(t-1)/t$. 
The obvious route to results is to condition on $Z$, use the formulas for 
the binomial random walk and integrate the power series of $T_t^m$ term-by term, 
using the Beta integrals. The general solution is given in the following proposition. Note 
that in the case $d>1$ we have to determine the probabilities 
$a_{m,d}=\PP(tdS_n-nds\leq m \mbox{ for all } n\in\mathbb{N})$.
Since clearly  $a_{m,d}=a_{\lfloor m/d\rfloor}$, (for $r\in\mathbb{R}$, $\lfloor r \rfloor$ denotes the largest integer not exceeding $r$) this can be reduced to the computation of the $a_m$, 
i.e., the case $d=1$, which we dealt with in Section 4.   
\begin{prp}
Let $m=b(t-1)-r\geq 0$ and $a=\lfloor m/d\rfloor$. Then

\begin{align*}
\PP(S_{r,b}>(t-1)/t)&=\sum_{n=0}^\infty \dfrac{a+1}{n(t-1)+a+1}\\
& \hspace{0.3cm} {nt+a \choose n} \dfrac{B(n(t-1)+a+1+(r/d),n+(b/d))}{B(r/d,b/d)}. 
\end{align*}
\end{prp}

For $d=1$ one gets series of rational functions of $n$, which can (in principle) 
be evaluated in closed form in terms of the Digamma function $\Psi$.  
We only give two examples for the results of these calculations:
\begin{prp}  
	For $d=1$ and $2\leq t \in\mathbb{N}$ we have 
\begin{equation*}\\
\begin{array}{ll}\PP(S_{1,1}&=(t-1)/t)=\dfrac{2t-3}{t}H(1-(1/t))-
\dfrac{t-2}{t}H(1-(2/t))\\
& \hspace{3.2cm} - \ \dfrac{t-2}{t-1}\\ \\ 
\PP(S_{1,1}&\le (t-1)/t)=(1-(1/t))H(1-(1/t)).  
\end{array}
\end{equation*}   
\end{prp}
Essentially this has already been shown by Knuth (\cite{knf5a}, Problem 88). 

The situation for $b=1,r=t-1$ is of special interest because in this case 
$p_+(t)=\PP(I_{t-1,1}\geq (t-1)/t) $ (the probability that the urn content process stays above 
 the line $y=(t-1)x$) as well as   
$p_-(t)=\PP(S_{t-1,1}\leq (t-1)/t) $ (the probability that it stays below the line $y=(t-1)x$) can 
both be considered.

Let $q_-(t)=1-p_-(t)=\PP(S_{t-1,1} >(t-1)/t)$. One finds that 
\begin{prp} For $d=1$, 
$$q_-(t) 
=(t-1)\sum_{n=0}^\infty\frac{(nt)!}{(nt-n+1)!}\frac{((n+1)(t-1))!}{((n+1)t)!}. 
$$ 
\end{prp}
Special values are 
\begin{itemize}
\item $q_-(2)=\ln 2\approx 0.6931 $
\item $q_-(3)=\dfrac{4}{27}\pi\sqrt{3}\approx 0.8061$ 
\item $q_-(4)=\dfrac{9}{32}\ln 2 +\dfrac{27}{128}\pi\approx 0.8576$ 
\item $q_-(5)\approx 0.8874, q_-(6)\approx 0.9068,\ldots, q_-(20)\approx 0.9726$. 
\end{itemize} 
Asymptotically we obtain $\lim_{t\to \infty} p_-(t)=0$, 
since clearly $q_-(t)\geq 1-(1/t)$ (the urn content process
steps from $(t-1,1)$ to $(t,1)$ with probability $1-(1/t)$). \\

Finally we look at the argument value $1/t$. In this case we have obtained the representation 
$$\PP(S_{1,t-1}\leq 1/t)=(t-1)\int_{0}^{1/t} (1-pt)q^{t-3}\,dp, $$
which can be evaluated elementarily. The result is given by 
\begin{prp}  For $d=1$, 
	\begin{align*} 
\PP(I_{t-1,1}\geq (t-1)/t)&=\PP(S_{1,t-1}\leq 1/t)\\&
=\left\{\begin{array}{ll} 1-\ln 2  &\mbox{ for } t=2\\
\big(\big[1-(1/t)\big]^{t-2}(t-1)-1\big)/(t-2) &\mbox { for } t>2 \end{array}\right. 
\end{align*} 
In particular, $p_+(t)\lra e^{-1}  \mbox { as } t \lra \infty $. 					  
\end{prp} 

It is also interesting to consider the other start positions $(a(t-1),a)$ ($2\leq a\in \mathbb{N}$)
on the line $y=(t-1)x$. Here one gets (again for $d=1$) 
	\begin{align*} 
\PP(I_{a(t-1),a}\geq (t-1)/t)&=\PP(S_{a,a(t-1)}\leq 1/t)\\
&=\int_0^{1/t} (1-pt)q^{-1}\,\beta_{a,a(t-1)}(p)\,dp.  
\end{align*} 
A short calculation shows that the latter 
integral can be expressed in terms of the binomial distribution. 
\begin{prp}  For $d=1$ and $2\leq a\in \mathbb{N}$, 
	\begin{align*} 
\PP(S_{a,a(t-1)}\leq 1/t)
=\PP(X_{at-1,1/t}=a)- \dfrac{1}{a(t-1)-1}  \PP(X_{at-1,1/t}>a)
\end{align*} 
where $X_{n,p}$ is a random variable having the 
binomial distribution with parameters $n$ and $p$. 
In particular, 
$$\PP(S_{a,a}\leq 1/2)=\big(1+\dfrac{1}{a-1}\big){2a-1 \choose a} 2^{-(2a-1)}-\dfrac{1}{2(a-1)} $$ 
 and 
$$\PP(S_{a(t-1),a}\leq 1/t)=O(a^{-1/2})\ {\rm as }\ a\lra \infty.$$ 
\end{prp}

In closing, we remark that for $t=2$ the ``equalization probability'' $\PP(S_{r,b}\geq 1/2)$ was already studied
in  \cite{ant} and \cite{wall}, where results equivalent to the ones above were obtained in this special case. 
In particular, for $t=2, m=b-r>0$ we get from \ref{ex1}
$$\PP(S_{r,b}\geq 1/2)=\int_0^{1}\min(1,p/q)^{b-r}\,\beta_{r,b}(p)\,dp=2\,\int_0^{1/2}\beta_{b,r}(p)\,dp,$$
yielding the identity 
$$\PP(S_{r,b}\geq 1/2)=2\,\PP(X_{b+r-1,1/2}\leq r-1),$$
which was shown in a different way in \cite{wall}.\\ 

\subsection*{NOTE ADDED IN PROOF} Finally we sketch how the general case can be treated. Let again $(s,t)$ be positive integers with $\gcd(s,t)=1$ and  introduce the generalized tree-function
$$T_{t/s}(z):=\sum_{n=0}^\infty \frac{z^n}{nt+1}{(nt+1)/s \choose n}\,.$$ 
Then $1/T_{t,s}(z)$ is the series  
$$\frac{1}{T_{t/s}(z)}=1-\sum_{n=1}^\infty \frac{z^n}{nt-1}{(nt-1)/s \choose n}\,.$$ 
and both series converge absolutely for $|z|\leq \tfrac{s}{t}(1-\tfrac{s}{t})^{t/s-1}$. If $0<p < s/t$ the roots of $pz^t-z^s+q$  in the closed unit disk are the values 
$$z_i(p):=\eta^iq^{1/s} T_{t/s}(pq^{(t-s)/s}\eta^{it}),\;i=1,\ldots, s$$
where $\eta$ is a primitive $s-$th root of unity (note that $z_s=1$), and the roots outside the unit disk are the values $w_i$ with $1/w_i(p)=y_i(p)$ where
$$y_i(p):=\omega^ip^{1/(t-s)} T_{t/(t-s)}(qp^{s/(t-s)}\omega^{it}),\;i=1,\ldots,t-s$$
and $\omega$ is a primitive $(t-s)$-th root of unity. Further, it is not hard to show (as in \cite{stad}) that 
$$a_0=a_0(p)=\prod_{i=1}^{t-s}(1-y_i(p))$$
and that 
$$g(z):=a_0+a_1z +\ldots + a_{s-1}z^{s-1}=a_0\prod_{i=1}^{s-1}(1-z/z_i)\;.$$ 
A similar argument as in the proof of \ref{prp2} shows that $g(1)=(s-tp)/q$.\\  If $m=bs-r(t-s)\geq 0$ we find that $$\PP(S_{r,b}\leq \frac{s}{t})=\int_0^1 a_m(p) \beta_{r,b}(p)\,dp$$ 
where the series for $a_m(p)$ can be (extracted from the information above and) integrated termwise.

\subsection*{Acknowledgements}
We would like to thank the referee for his very careful reading of the manuscript and for helpful comments. Special thanks go to Don Knuth
for a question that led to the note added in proof.

\end{document}